\documentclass[letterpaper, 10 pt, journal, twoside]{ieeetran} 
\usepackage{times}
\usepackage{amsthm}
\usepackage{mathrsfs}
\usepackage{amsmath} 
\usepackage{mathtools} 
\usepackage{graphicx,amsmath,amssymb,algorithmic}
\usepackage{cite}
\usepackage{dsfont}
\usepackage{times}
\usepackage{mathrsfs}
\usepackage{psfrag}
\usepackage[colorlinks=true, urlcolor=blue, citecolor=blue, linkcolor=blue]{hyperref}
\usepackage{mathtools} 
\usepackage{graphicx,amsmath,amssymb,algorithmic,color}
\usepackage{cite}
\usepackage{tikz} 
\usepackage{float}
\usepackage{epstopdf}
\usepackage{dsfont}

\newcommand{\sign}{\text{sign}}
\newtheorem{nncorollary}{\bf Corollary}
\newenvironment{corollary}{\begin{nncorollary}\it}{\hfill \hspace*{1pt}\hfill $\diamond$ \end{nncorollary}}
\newtheorem{nntheorem}{\bf Theorem}
\newenvironment{theorem}{\begin{nntheorem}\it}{\hfill \hspace*{1pt}\hfill $\diamond$\end{nntheorem}}
\newtheorem{nndefinition}{\bf Definition}
\newenvironment{definition}{\begin{nndefinition}\it}{\hfill \hspace*{1pt}\hfill $\diamond$\end{nndefinition}}
\newtheorem{nnproposition}{\bf Proposition}
\newenvironment{proposition}{\begin{nnproposition}\it}{\hfill \hspace*{1pt}\hfill $\diamond$\end{nnproposition}}
\newtheorem{nnproblem}{\bf Problem}
\newenvironment{problem}{\begin{nnproblem}\it}{\hfill \hspace*{1pt}\hfill $\diamond$\end{nnproblem}}
\newtheorem{nnlemma}{\bf Lemma}
\newenvironment{lemma}{\begin{nnlemma}\it}{\hfill \hspace*{1pt}\hfill $\diamond$\end{nnlemma}}
\newtheorem{nnremark}{\bf Remark}
\newenvironment{remark}{\begin{nnremark}}{\hfill \hspace*{1pt}\hfill $\circ$\end{nnremark}}

\def\ran{\mathtt{Ran}}
\def\dom{\mathtt{dom}}

\usepackage{mathtools}

\DeclarePairedDelimiter\abs{\lvert}{\rvert}%
\makeatletter
\let\oldabs\abs
\def\abs{\@ifstar{\oldabs}{\oldabs*}}
\newcommand{\norm}[1]{\left\lVert#1\right\rVert}
\newcommand{\source}{{THIS IS A PREPRINT VERSION. IF YOU FOUND THIS READING USEFUL FOR YOUR RESEARCH PLEASE CITE THE PUBLISHED VERSION DOI: \href{https://doi.org/10.1016/j.automatica.2022.110346}{https://doi.org/10.1016/j.automatica.2022.110346}}}
 \usepackage{fancyhdr}
\makeatletter

\def\ps@IEEEtitlepagestyle{}
\title{\LARGE \bf 
Design of Saturated Boundary Control for Hyperbolic Systems with In-domain Disturbances}

\author{Suha Shreim, Francesco Ferrante, and Christophe Prieur 
\thanks{S. Shreim and C. Prieur Univ. Grenoble Alpes, CNRS, Grenoble INP, GIPSA-lab, 38000 Grenoble, France. Email: \{suha.shreim, christophe.prieur\}@gipsa-lab.fr}
\thanks{Francesco Ferrante is with Department of Engineering, University of Perugia, Via G. Duranti, 67, 06125 Perugia, Italy. Email: francesco.ferrante@unipg.it}
\thanks{This work has been partially supported by MIAI@Grenoble Alpes (ANR- 19-P3IA-0003)}
}
\begin{document}        
 \maketitle                                       
\begin{abstract}            
Boundary feedback control design is studied for 1D hyperbolic systems with an in-domain disturbance and a boundary feedback controller under the effect of actuator saturation. Nonlinear semigroup theory is used to prove well-posedness of mild solution pairs to the closed-loop system. Sufficient conditions in the form of dissipation functional inequalities are derived to establish global stability for the closed-loop system and $\mathcal{L}^2$-stability in presence of in-domain disturbances. The control design problem is then recast as an optimization problem over linear matrix inequality constraints. Numerical results are shown to validate the effectiveness of the proposed control design.\end{abstract}
Partial differential equations are mathematical expressions which are found to be of great importance in the modeling of many physical systems that are described simultaneously via spatial and temporal variables. Light propagation
	in optic fibers, blood flow in the vessels, plasma
	in laser, liquid metals in cooling systems, road traffic, acoustic waves, and electromagnetic waves are all examples of systems modeled via PDEs that can be seen in civil, nuclear, mechanical, quantum, and chemical engineering (see
	\cite{1} and \cite{3} for more examples). Thus, the importance of studying the analysis and control of physical systems modeled via PDEs is growing more and more in the community of automatic control. 
	In fact, what renders this topic challenging is the infinite-dimensional nature of the application. In the majority of the cases, actuators and sensors are located in the boundaries of the system. In other words, in-domain control, which while theoretically facilitates the stability analysis and control of PDEs, is impractical in the real world. Thus, we can find a lot of research conducted on the topic of boundary control \cite{3}. In particular, different control strategies, from Lyapunov stability and stabilization methods \cite{1},\cite{trinh2017design}, to backstepping control \cite{russell1978controllability}, \cite{3} and frequency domain approaches \cite{litrico2009modeling} have been applied on PDEs.
	
	Another major problem that faces control engineers is the presence of actuator saturation threatens the control system. Neglecting the presence of this saturation can be a source of undesirable and dangerous behaviors in the closed-loop system \cite{5}. That is why, if one wishes to conclude a realistic, safe control design, it is of utmost necessity to take into account this saturation within the system modeling. Unfortunately, this means that we are adding a nonlinearity to our model. Researchers have been studying several methods to tackle saturation problems in closed-loop systems as we can see in \cite{2}, \cite{zaccarian2011modern} or \cite{slemrod1989feedback} and some examples of extensions of those works are present in \cite{seidman2001note} and \cite{marx2017cone} . Stability analysis of PDEs in the presence of saturation has been studied in the math community {\cite{haraux2018nonlinear}, \cite{alabau2012some}}, but it is still an open research area, especially from an automatic control viewpoint {\cite{xu2019saturated}, \cite{jacob2020remarks}}. {To the best of our knowledge, the particular problem of designing a boundary controller under the effect of saturation to stabilize hyperbolic systems has not been studied in previous works}. As we have seen in \cite{6}, a natural approach to study the stability problem is to combine both Lyapunov theory and cone-bounded sector conditions (see more about sector conditions in \cite{8},\cite{2}, \cite{castelan2008control}).\\
	Studies on the well-posedness of infinite dimensional systems in the presence of nonlinearities have been presented in \cite{barbu2010nonlinear} \cite{tucsnak2009observation}, and \cite{barbu1976nonlinear} based on semigroup theory. In \cite{wellposed}, the authors make use of infinite dimensional linear systems theory to rewrite a linear PDE and interconnect it with a static nonlinearity. Only few papers study the well-posedness of hyperbolic PDEs in the presence of saturated boundary nonlinearity. More specifically, \cite{6} considers the wave equation, whereas \cite{dus2020stabilization} analyses the stability of BV solutions.\\
	In this work, we focus on systems of one-dimensional conservation laws modeled as a system of linear hyperbolic PDEs. The aim of this article is to study the class of hyperbolic systems in the presence of nonlinear control laws \cite{shreim2020design} as well as an in-domain exogenous disturbance \cite{ferrante2020boundary}. In this paper, we consider a classical saturation function and we present a systematic approach to design static boundary controllers to ensure closed-loop exponential stability and robustness with quantifiable margins with respect to in-domain energy bounded disturbances. To {achieve} this goal, we first prove the existence and uniqueness of solution to our closed-loop system. Then, we prove the global exponential stability by using a sector condition and a suitable Lyapunov functional. The proposed conditions are embedded into a convex optimization setup to enable the design of a controller minimizing the effect of the disturbance on the closed-loop system.\\
	The paper is organized as follows. Section \ref{ProblemStatement} illustrates the problem we solve and defines the notion of solution we use.
	Section \ref{wellposednessandReg} tackles the well-posedness of the closed-loop system using properties of non-accretive operators. Section \ref{stabilityanalysis} is dedicated to Lyapunov analysis and provides the sufficient conditions for stability. Furthermore, {it} presents the control design problem in the sense of an optimization problem which gives an optimal control gain. Section \ref{numerical} validates the effectiveness of the proposed design algorithm through a numerical example. 
	\subsection{Notation}
	The sets $\mathbb{R}_{\geq 0}$ and $\mathbb{R}_{> 0}$ represent the set of nonnegative and positive real scalars, respectively. The symbols $\mathbb{S}^n_p$ and $\mathbb{D}^n_p$ denote, respectively, the set of real $n\times n$ symmetric positive definite matrices and the set of diagonal positive definite matrices. For a matrix $A \in \mathbb{R}^{n\times m}$, $A^\top$ denotes the transpose of $A$ and $\norm{A}$ denotes the 2-induced matrix-norm of $A$. For a symmetric matrix $A$, positive
	definiteness (negative definiteness) and positive semidefiniteness
	(negative semidefiniteness) are denoted, respectively, by $A>0$ $(A<0)$ and $A\geq 0$ $(A\leq 0)$. In partitioned symmetric matrices, the symbol $*$ stands for symmetric blocks. For a vector $x\in \mathbb{R}^n$, $\vert x \vert$ denotes its Euclidean norm. 
	Let $X$ and $Y$ be normed linear spaces, the symbol $\mathcal{L}(X, Y)$ 
	denotes the space of all bounded linear operators from $X $ to
	$Y $. For $U\subset \mathbb{R}$, we denote by {$\norm{f}_{\mathcal{L}^2(U)} = (\int_U \vert f(z)\vert^2 dz)^{\frac{1}{2}}$}, the $\mathcal{L}^2$-norm of \footnote{In this paper, we only consider Lebesegue measurable functions.}$f$. Given $f : U \subset \mathbb{R} \xrightarrow{} V$, we
	say that $f\in \mathcal{L}^2$ if $f$ is measurable and $\norm{f}_{\mathcal{L}^2{(U)}}$ is finite.
	For an open $U\subset \mathbb{R}$ and a normed linear space  $V\subset \mathbb{R}^n$, $\mathcal{H}^1(U,V):= \{f\in \mathcal{L}^2(U,V): f$ is locally absolutely continuous on $U, \frac{d}{dz}f\in \mathcal{L}^2(U,V)\}$ where $\frac{d}{dz}$ stands for the weak derivative of $f$. The symbol $Df(x)$ indicates the Fréchet derivative (when it exists) of the function $f$ at $x$.The symbol $\mathcal{C}^1(U,V)$ denotes the set of functions  $f: U\xrightarrow[]{} V$ that are continuously
	differentiable. Moreover, the symbol $\mathcal{C}^\infty_c(U,V)$  denotes the set of smooth compactly supported functions $f: U\xrightarrow[]{} V$.

	\section{Problem Statement}\label{ProblemStatement}
	\subsection{Problem setup}
	We consider the boundary feedback control of the following $n$ linear 1-D hyperbolic PDEs formally written as: 
	\begin{equation}\label{systemmm}
		\begin{array}{ll}
			X_t(t,z) + \Lambda X_z(t,z)= N d(t,z) & \forall (t,z) \in \mathbb{R}_{\geq 0}\times(0,1)   \\
			X(t,0)= H X(t,1)+ B \sigma(u(t)) & \forall t \in \mathbb{R}_{\geq 0}\\
			X(0,z)= X_0(z) & \forall z \in (0,1)
		\end{array}
	\end{equation} 
	where $t \in \mathbb{R}_{\geq 0}$ and $z \in (0,1)$ are the two independent variables, respectively, time and space, $ z \mapsto X(.,z) \in \mathbb{R}^n$ is the state, and $z\mapsto d(.,z)\in \mathbb{R}^q$ is an exogenous in-domain disturbance. We assume also that the matrices $\Lambda \in \mathbb{D}^n_p$, $H\in \mathbb{R}^{n\times n}$, $B\in \mathbb{R}^{n\times m}$ and $N\in \mathbb{R}^{n\times q}$ are given and that the state $X(\cdot,z)$ is measurable only at the boundary point $z = 1$. Specifically, the measurable output of the system reads as $y= X(\cdot,1)$.\\
	Let $u:= K X(\cdot,1)$ where $K\in \mathbb{R}^{m\times n}$ is the control gain to be designed and the function $u\mapsto \sigma (u)$ is the symmetric decentralized saturation function with saturation levels $\overline{u}_1,\overline{u}_2, \dots, \overline{u}_{m}\in\mathbb{R}_{>0}$, whose components for each $u\in\mathbb{R}^m$ are defined as:
	\begin{equation*}
		\sigma(u)_i=\sigma(u_i):= \min(\vert u_i\vert, \overline{u}_{i})\sign (u_i)\quad i=1,2,\dots, m
	\end{equation*}
	Our goal is to design the gain $K$ to induce closed-loop stability with quantifiable convergence rate and robustness margins with respect to the exogenous input $d$. 
	For convenience, we define the function $u \mapsto \phi (u)$ which is the symmetric decentralized dead-zone nonlinearity function given by the following expression (see \cite[page 40]{2}):\\
	\begin{equation}
		\phi(u_i) := 
		\sigma(u_i)-u_i
	\end{equation} 
	where $\phi : \mathbb{R}^m \xrightarrow[]{} \mathbb{R}^m$.
	By setting $H_{cl}:= H+BK$, the closed-loop system turns into:
\small\begin{equation}\label{cls}
	\begin{array}{ll}
		X_t(t,z) + \Lambda X_z(t,z)= Nd(t,z) & \forall (t,z) \in \mathbb{R}_{>0}\times(0,1)   \\
		X(t,0)= H_{cl}X(t,1) + B\phi(KX(t,1)) & \forall t \in \mathbb{R}_{\geq 0}\\
		X(0,z)= X_0(z) & \forall z \in (0,1)
	\end{array}
\end{equation}
\begin{remark}
   {Even though we only study the homodirectional case with $\Lambda \in \mathbb{D}_p^n$, the results can be directly extended to a heterodirectional case with $\Lambda = \begin{pmatrix}
        \Lambda_+\\ \Lambda_-
    \end{pmatrix}$ where $\Lambda_+ \in \mathbb{D}_p^r$, $\Lambda_- \in \mathbb{D}_p^s$ and $s+r=n$. This is done through a change in variable.}
\end{remark}

	\subsection{Notion of solution}
	Similarly as in \cite{barbu2010nonlinear}, we focus on mild solution pairs to (\ref{systemmm}). As in \cite{curtain2012introduction}, we start by reformulating the closed-loop system as an abstract differential equation. 
	Consider now the following operators defined, respectively, on the Hilbert spaces $\mathcal{L}^2(0,1; \mathbb{R}^n)$ and $\mathcal{L}^2(0,1;\mathbb{R}^q)$ equipped with their respective standard inner products:
	\begin{equation}
		\begin{split}
			\mathcal{A}:& \mathcal{D}(\mathcal{A})\xrightarrow[]{} \mathcal{L}^2(0,1;\mathbb{R}^n)  \\
			&X\mapsto-\Lambda X_z(z)\\
			\mathcal{N}:& \mathcal{L}^2(0,1;\mathbb{R}^q) \xrightarrow[]{} \mathcal{L}^2(0,1;{\mathbb{R}^n})\\
			&d\mapsto N d
		\end{split}
	\end{equation}
	where
	$$\mathcal{D}(\mathcal{A}):= \{X\in \mathcal{H}^1(0,1;\mathbb{R}^n); X(0)=H_{cl}X(1)+B\phi(KX(1))\}$$
	Then, the closed-loop dynamics can be
	formally written as the following abstract system with state $X\in \mathcal{L}^2(0,1;\mathbb{R}^n)$ and exogenous input $d\in \mathcal{L}^2(0,1;\mathbb{R}^q)$
	\begin{equation}\label{abstract}
			\dot {X} = \mathcal{A} X + \mathcal{N} d		
	\end{equation} 
	In particular, we use the following notion of a mild solution pair for (\ref{abstract}). We recall the definition introduced in \cite[Definition 4.3, page 120]{barbu2010nonlinear}:
	\begin{definition}
		A mild solution pair for the system (\ref{abstract}), with the initial condition {$X(0,z)=X_0$} is a pair $(X,d)$ satisfying the following: functions $X\in \mathcal{C}(\dom X;\mathcal{L}^2(0,1;\mathbb{R}^n))$ and {$d\in \mathcal{L}^1(\dom d,\mathcal{L}^2(0,1;\mathbb{R}^q))$} where \footnote{{$\dom d$ and $\dom X$ can be either $[0,T]$ or $[0, \infty)$.}}$\dom X= \dom d$ is an interval of $\mathbb{R}_{\geq 0} $ including zero and for each $\epsilon >0$, there exists an $\epsilon$-approximate solution $z$ (using the terminology of \cite[Definition 4.2, page 129]{barbu2010nonlinear})  of $\dot {X} = \mathcal{A} X + \mathcal{N} d$ such that $\norm{X(t)-z(t)}\leq \epsilon$ for all $t\in \dom X$.
		
	\end{definition}
{In what follows, we use the shorthand notation $\mathcal{L}^2_n$ instead of $\mathcal{L}^2(0,1;\mathbb{R}^n)$ and, similarly, $\mathcal{L}^2_q$ instead of $\mathcal{L}^2(0,1;\mathbb{R}^q)$.} Now we state the problem we solve in this paper:
	\begin{problem}\label{solutionproblem}
		Given $H\in \mathbb{R}^{n\times n}, B\in \mathbb{R}^{n\times m}$, $N\in \mathbb{R}^{n\times q}$, and $\Lambda \in \mathbb{D}^n_p$. Design $K$ such that for some $\kappa$, $\omega$, $\gamma \in \mathbb{R}_{>0}$ and for each mild solution pair (X,d) to (\ref{abstract}) one has, for all $t\in \dom X$:
		\begin{equation}\label{lmiprob}
			\norm{X(t)}_{{{\mathcal{L}^2_n}}} \leq \kappa e^{-\omega t}\norm{{X_0}}_{{{\mathcal{L}^2_n}}} + \gamma \sqrt{\int^t_0 \norm{d(\theta)}^2_{{{\mathcal{L}^2_q}}} d\theta }
		\end{equation} 
	\end{problem}
{Inequality (\ref{lmiprob}) corresponds to a classical input-to-state-stability (ISS) bound for the abstract closed-loop system (\ref{abstract}). Sufficient conditions to ensure ISS for infinite dimensional systems are given in \cite{karafyllis2019input} and \cite{mironchenko2020input}}. The main contribution of this paper is to perform an optimal design of the control design of the control gain $K$ in order to minimize the ISS gain $\gamma$. In Section \ref{stabilityanalysis}, we provide sufficient conditions to get an explicit estimate of the ISS gain $\gamma$.
	\section{Well-posedness of the Closed-Loop System}\label{wellposednessandReg}
	In this section, we state the well-posedness of the closed-loop system (\ref{abstract}). Let us start by defining the notion of non-accretive operator inspired by \cite[Definition 3.1, page 97]{barbu2010nonlinear}:
	\begin{definition} 
		An operator $\mathcal{A}$ from $\mathcal{D}(\mathcal{A})$ to $\mathcal{L}^2(0,1;\mathbb{R}^n)$ is said to be non-accretive with respect to an inner product $ \langle\cdot ,\cdot \rangle$ if for every pair $(X_1, X_2)\in \mathcal{D}(\mathcal{A}) \times \mathcal{D}(\mathcal{A})$, the following inequality holds:
		\begin{equation}\label{accret}
			\langle\mathcal{A} X_1-\mathcal{A} X_2,X_1-X_2\rangle \leq 0
		\end{equation}
	\end{definition}
	Inspired by \cite[Appendix A, page 224]{1}, let us introduce the following  inner product on $\mathcal{L}^2(0,1;\mathbb{R}^n)$:
	\begin{equation}\label{Lmuspace}
		\langle X_1,X_2\rangle_{\mu}:= \int_0^1 e^{\mu z}X_1^\top  X_2 dz
	\end{equation}
	where $\mu>0$ will be selected later. {It is noted that this inner product is equivalent to the standard inner product in $\mathcal{L}^2$ since the function $z\mapsto e^{\mu z}$ is bounded from below and above on $[0,1]$}. We now use the previous definition to apply it on a suitable operator which will be vital in proving the uniqueness and {existence} of mild solution pairs to (\ref{abstract}).
	
	\begin{proposition}\label{accretprop}
		There exist $\mu>0$ and ${\rho} \in \mathbb{R}$ such that the operator $\mathcal{A}+{\rho} I$ is non-accretive (with respect to the scalar product 
		$\langle \cdot,\cdot\rangle_\mu$).
	\end{proposition}
	\begin{proof}
		Let $X_1, X_2 \in \mathcal{D(A)}$, $\Tilde{X}=X_1-X_2 \in \mathcal{D(A)}$ and $\Tilde{\phi}=\phi(KX_1(1))-\phi(KX_2(1))\in \mathbb{R}^m$. Let us prove (\ref{accret}) for the operator $\mathcal{A}+ {\rho} I$ for a suitable choice of $\mu$. First we can write the following:
		$$\langle(\mathcal{A}+{\rho} I)\Tilde{X},\Tilde{X}\rangle_\mu= \langle\mathcal{A}\Tilde{X},\Tilde{X}\rangle_{\mu} +{\rho} \langle\Tilde{X}, \Tilde{X}\rangle_{\mu}$$
		where $$ \langle\mathcal{A}\Tilde{X},\Tilde{X}\rangle_{\mu}= -\int^1_0 e^{\mu z} \Tilde{X}^\top \Lambda \Tilde{X}_z dz $$
		Using an integration by parts, we have: 
		$$\langle\mathcal{A}\Tilde{X},\Tilde{X}\rangle_{\mu}= -\frac{1}{2} e^{\mu z}\Tilde{X}^\top \Lambda \Tilde{X} \Big|^1_0 +\frac{1}{2} \mu \int^1_0 e^{\mu z} \Tilde{X}^\top \Lambda \Tilde{X} dz$$ 
		Thanks to the boundary condition in (\ref{cls}), we have: 
		\begin{equation}\label{tilde}
			\begin{split}
				\langle\mathcal{A}\Tilde{X}, \Tilde{X}\rangle_\mu = -\frac{1}{2} &\left( e^{\mu}\Tilde{X}(1)^\top\Lambda \Tilde{X}(1) -(H_{cl}\Tilde{X}(1)+B\Tilde{\phi})^\top\right. \\&\left. 
				\Lambda (H_{cl}\Tilde{X}(1)+B\Tilde{\phi})\right) +\frac{1}{2}\mu  \langle\Tilde{X}, \Tilde{X}\rangle_{\mu}\end{split}
		\end{equation}
		We can rewrite the previous equation as: 
		\begin{equation}\label{eq:4:12}
			\begin{split}
				\langle\mathcal{A}\Tilde{X}, \Tilde{X}\rangle_{\mu}=&
				\frac{1}{2}\mathcal{X}^\top \begin{pmatrix}  H_{cl}^\top \Lambda H_{cl}- e^{\mu}\Lambda & H_{cl}^\top \Lambda B\\
					* & B^\top \Lambda B
				\end{pmatrix}    
				\mathcal{X}\\&+\frac{1}{2}\mu  \langle\Tilde{X}, \Tilde{X}\rangle_\mu
			\end{split}
		\end{equation} where $\mathcal{X}\coloneqq \begin{pmatrix}
		\Tilde{X}(1)\\\Tilde{\phi}
	\end{pmatrix}$. 
		Hence, recalling that $\phi$ is 1-Lipschitz continuous {\footnote{{This follows from the fact $\phi$ is continuous and piecewise linear with a slope bounded by $1$.}}}, one has:
		$$
		\mathcal{X}^\top
		\begin{pmatrix}-K^\top K & & 0 \\ * & & I\end{pmatrix}
		\mathcal{X}\leq 0
		$$
		which by using, (\ref{eq:4:12}) gives:
		\begin{equation}
			\begin{split}
				\langle\mathcal{A}\Tilde{X}, \Tilde{X}\rangle_\mu \leq&   \frac{1}{2}\mathcal{X}^\top  
				\begin{pmatrix}
					H_{cl}^\top \Lambda H_{cl}- e^{\mu}\Lambda +\tau K^\top K& H_{cl}^\top \Lambda B\\
					* & B^\top \Lambda B-\tau I
				\end{pmatrix}{\mathcal{X}} \\ &+\frac{1}{2}\mu  \langle\Tilde{X}, \Tilde{X} \rangle_\mu
			\end{split}
		\end{equation}
		for any $\tau>0$. Pick $\tau$ such that $B^\top \Lambda B - \tau I  \leq -I $.
		Thus, we can write
		\begin{equation}
			\begin{split}	\langle\mathcal{A}\Tilde{X}, \Tilde{X}\rangle_\mu \leq&  \frac{1}{2}\mathcal{X}^\top \begin{pmatrix}
					H_{cl}^\top \Lambda H_{cl}- e^{\mu}\Lambda +\tau K^\top K & H_{cl}^\top \Lambda B\\
					* & -I
				\end{pmatrix}\mathcal{X}\\ &+\frac{1}{2}\mu  \langle\Tilde{X}, \Tilde{X}\rangle_\mu
			\end{split}
		\end{equation}
		Now, consider the following matrix:
		\begin{equation*}
			\Omega  := \begin{pmatrix}
				H_{cl}^\top \Lambda H_{cl}- e^{\mu}\Lambda +\tau K^\top K & H_{cl}^\top \Lambda B\\
				* & -I
		\end{pmatrix}\end{equation*} 
		From the Schur-complement lemma {(see \cite[page 34, Theorem 1.12]{zhang2006schur})}, one has that $\Omega < 0$ if and only if the following conditions hold 
		\begin{equation}\label{omegacond}
			\begin{array}{c}
			H_{cl}^\top \Lambda H_{cl}- e^{\mu}\Lambda +\tau K^\top K < 0 \\
				-I - (B^\top \Lambda H_{cl}) (H_{cl}^\top \Lambda H_{cl}- e^{\mu}\Lambda +\tau K^\top K)^{-1} (H_{cl}^\top \Lambda B)< 0
			\end{array}
		\end{equation} 
		Pick $\mu$ such that:
		\begin{equation}
	\mu > \ln \left (\norm{H_{cl}^\top \Lambda H_{cl}\Lambda^{-1} + \tau K^\top K\Lambda^{-1} + \norm{H_{cl}^\top \Lambda B}^2 \Lambda^{-1} }\right )
\end{equation}

 Thus, both conditions of (\ref{omegacond}) hold and $\Omega < 0$. Finally, choose ${\rho} <-\frac{1}{2} \mu$ and thus, (\ref{accret}) holds and the proof is concluded.
\end{proof}

Following the work of {\cite[page 97]{barbu2010nonlinear}}, we now prove that the non-accretive operator $\mathcal{A}+{\rho} I$ enjoys the following property:
	\begin{proposition}\label{rangeprop}
		There exists ${\rho}\in \mathbb{R}$ such that for all $\lambda>0$, the following range property holds 
		\begin{equation}\label{rangecomp}
			\ran(I+\lambda(\mathcal{A}+{\rho} I)) = \mathcal{L}^2(0,1;\mathbb{R}^n)
		\end{equation}
		where $\ran$ stands for the range.
	\end{proposition}
	\begin{proof}
		We know that $$\ran(I+\lambda(\mathcal{A}+{\rho} I)) \subset \mathcal{L}^2(0,1;\mathbb{R}^n) $$
		Let us prove that \begin{equation}\label{range}
			\ran(I+\lambda(\mathcal{A}+{\rho} I)) \supset \mathcal{L}^2(0,1;\mathbb{R}^n)
		\end{equation}
		Pick any $f\in \mathcal{L}^2(0,1;\mathbb{R}^n)$, we show that there exists $X \in\mathcal{D}(\mathcal{A})$ such that $$(I+\lambda(\mathcal{A}+{\rho} I))X=f$$
		The above statement is equivalent to checking the existence of solution to the following boundary value problem:
		\begin{equation}\label{ode}
			\begin{array}{lc}
				I_{{\rho}} X(z){-}\lambda \Lambda X_z(z)= f(z)  &  \forall z\in (0,1) \\
				X(0)=H_{cl}X(1)+B\phi(KX(1))
			\end{array}
		\end{equation} 
		where $I_{{\rho}}:= (1 +\lambda {\rho}) I$.
		The solution for the first line of (\ref{ode}) is given by:
		{\begin{equation}\begin{aligned}
				X(z)= e^{\frac{1}{\lambda} \Lambda^{-1} I_{{\rho}} z}X(0){-}\int^z_0 e^{\frac{1}{\lambda} \Lambda^{-1}I_{\rho}(z-s)}\frac{1}{\lambda} \Lambda^{-1}f(s) ds\\ \forall z\in (0,1)
			\end{aligned}
		\end{equation}} In particular, one has:
		{\begin{equation*}\begin{split}
					X(1)=& e^{\frac{1}{\lambda} \Lambda^{-1}I_{{\rho}}} X(0){-}\int_0^1 e^{\frac{1}{\lambda} \Lambda^{-1}I_{\rho}(1-s)}\frac{1}{\lambda} \Lambda^{-1}f(s)ds\\=\colon& g_\lambda (X(0))
		\end{split}\end{equation*}}
		Then, the boundary condition is rewritten as:
		\begin{equation}\label{BCmodif}
			X(0)= H_{cl}g_\lambda(X(0))+B\phi(Kg_\lambda(X(0))
		\end{equation}
		Therefore, (\ref{ode}) has a solution if and only if there exists $X(0)$ satisfying (\ref{BCmodif}). Let us introduce the following map:
		\begin{equation}\label{map}
			\begin{split}
				\mathcal{T}\colon& \mathbb{R}^n \xrightarrow[]{} \mathbb{R}^n\\
				&c \mapsto H_{cl}g_\lambda(c)+B\phi(Kg_\lambda(c))
			\end{split}
		\end{equation}
		Now, we show that this is the case by using Banach fixed point theorem {\cite[page 138]{brezis2010functional}} to $ \mathcal{T} $. In order to show that $ \mathcal{T} $ is a contraction, let us first write that
		{\begin{equation}\label{gminus}
			g_\lambda(c_1)-g_\lambda(c_2)=  e^{\frac{1}{\lambda} \Lambda^{-1}I_{{\rho}}} (c_1 - c_2) \qquad \forall c_1, c_2 \in \mathbb{R}^n
		\end{equation} }
		Since $\phi$ is a 1-Lipschitz continuous function, it follows that for all $ c_1, c_2 \in \mathbb{R}^n $:
		$$
			\abs*{\phi(Kg_\lambda(c_1))-\phi(Kg_\lambda(c_2))}\leq \abs*{K(g_\lambda(c_1)-g_\lambda(c_2))} 
		$$
		which, using (\ref{gminus}), gives: 
		\begin{equation}\label{phii}
				\abs*{\phi(Kg_\lambda(c_1))-\phi(Kg_\lambda(c_2))} \leq  \abs*{K(e^{\frac{1}{\lambda} \Lambda^{-1}I_{{\rho}}} (c_1 - c_2))}
		\end{equation}
		Using (\ref{map}), (\ref{gminus}), and (\ref{phii}), we get:{
		\small\begin{equation*}
			\begin{split}
			&\abs*{\mathcal{T}(c_1)- \mathcal{T}(c_2)}\\&= \abs*{H_{cl}g_\lambda(c_1)+B\phi(Kg_\lambda(c_1))- H_{cl}g_\lambda(c_2)-B\phi(Kg_\lambda(c_2))}\\&\leq
			\abs*{H_{cl}(g_\lambda(c_1)-g_\lambda(c_2))} +\abs*{B\phi((Kg_\lambda(c_1))-\phi(Kg_\lambda(c_2)))}\\ &\leq
			\abs*{ H_{cl}(e^{\frac{1}{\lambda} \Lambda^{-1} I_{{\rho}}}(c_1 - c_2))}+\abs*{BK(e^{\frac{1}{\lambda} \Lambda^{-1} I_{{\rho}}} (c_1 - c_2))}\\
			&\leq \left(\norm {H_{cl}e^{\frac{1}{\lambda} \Lambda^{-1} I_{{\rho}}}}   +\norm{BKe^{\frac{1}{\lambda} \Lambda^{-1} I_{{\rho}}}}\right)\abs*{c_1-c_2}
		\end{split} \end{equation*}}
		
		Let $\alpha= \norm{H_{cl}e^{\frac{1}{\lambda} \Lambda^{-1} I_{{\rho}}}} +\norm{BKe^{\frac{1}{\lambda} \Lambda^{-1} I_{{\rho}}}}$. 
		We have:
		\begin{equation}
			\begin{split}
				\alpha \leq& (\norm{H_{cl}}+\norm{ BK})\norm{e^{\frac{1}{\lambda} \Lambda^{-1} I_{{\rho}}}}\\\leq& e^{\frac{\lambda {\rho}}{\lambda_{\max}(\Lambda)\lambda}}(\norm{H_{cl}} +\norm{BK})
			\end{split}
		\end{equation}
	where {$\lambda_{\max}(\Lambda)$ is the largest eigenvalue of the matrix $\Lambda$.}
	Pick ${\rho}\in \mathbb{R}$ small enough, such that {$e^{\frac{ {\rho}}{\lambda_{\max}(\Lambda)}}(\norm{H_{cl}} +\norm{BK}) < 1$}. Then, we have that $ 0 <\alpha< 1 $. The proof is concluded
	\end{proof}
	
	The main result of this section is presented in the following theorem where we show that the system is well-posed.
	\begin{theorem}
		\label{CrandallLigget}
		For every initial state $X_0\in \mathcal{L}^2(0,1;\mathbb{R}^n), d\in \mathcal{L}^1(\dom d;\mathcal{L}^2(0,1;\mathbb{R}^{q}))$, the closed-loop system (\ref{cls}) admits a unique mild solution pair $(X,d) \in \mathcal{C}(\dom X;\mathcal{L}^2(0, 1;\mathbb{R}^n)) \times \mathcal{L}^1(\dom d;\mathcal{L}^2(0,1;\mathbb{R}^{q}))$ such that {$X(0,z)=X_0$}.
	\end{theorem}
	\begin{proof}
		The choice of $X_0 \in \mathcal{L}^2_{\mu}(0,1;\mathbb{R}^n)$ is equivalent to $X_0 \in \mathcal{L}^2(0,1;\mathbb{R}^n)$ where $\mathcal{L}^2_{\mu}$ is defined by the norm induced by the scalar product in (\ref{Lmuspace}). By means of Propositions \ref{accretprop} and \ref{rangeprop}, the operator $\mathcal{A}+{\rho} I$ is $m$-non-accretive and thus, the Cauchy problem (\ref{cls}) has a unique mild solution pair, (see \cite[Theorem A.26, page 286]{andreu2004parabolic} and \cite[page 97]{barbu2010nonlinear}).
	\end{proof}
	\begin{definition}
		In \cite[page 127]{barbu2010nonlinear}, a strong solution pair to (\ref{abstract}) is defined as a pair $(X,d)\in (W^{1,1}(\dom X ;\mathcal{L}^2(0,1;\mathbb{R}^n))\cap \mathcal{C}(\dom X;\mathcal{L}^2(0,1;\mathbb{R}^n))) \times \mathcal{L}^1(\dom d; \mathcal{L}^2(0,1;\mathbb{R}^q))$ such that 
		\begin{equation*}
			\begin{split}
				& \frac{dX}{dt}(t)+\mathcal{A} X(t) =\mathcal{N}d(t)  \qquad t\in \dom X,\\& {X(z,0)=X_0}
			\end{split}
		\end{equation*}
		where $ X_0 \in \mathcal{L}^2(0,1;\mathbb{R}^n)$.
	\end{definition}
 In the remaining part of this section, we restrict the focus on the perturbation $d\in\mathcal{L}^2(\dom d;\mathcal{L}^2(0,1;\mathbb{R}^n))$ which is instrumental for the derivation of stability results of Section \ref{stabilityanalysis}. { The following proposition is crucial for the stability analysis Section 4.1, in which the mild solution pair $(X,d)$ can be approximated point-wise via a sequence of strong solution pairs (\ref{abstract}).}
	\begin{proposition} \label{regularity}
		Let $(X,d)$ be a mild solution pair to (\ref{abstract}) and $t\in \dom X$. There exists {a sequence of} strong solution pair $\{(X^k, d^k)\}_{k\in \mathbb{N}}$ such that: 
		\begin{equation}\label{xseq}
			X^k(t) \xrightarrow[k\xrightarrow{}\infty]{\mathcal{L}^2(0,1;\mathbb{R}^n)} X(t)
		\end{equation}
		\begin{equation}\label{dseq}
			d^k \xrightarrow[k\xrightarrow{}\infty]{\mathcal{L}^2(0,t;\mathcal{L}^2(0,1;\mathbb{R}^q))} d
		\end{equation}
	with $\dom X^k= ]0,t]$ and for all $k \in \mathbb{N}$.
	\end{proposition}
	\begin{proof}
		{Let $(X,d)\in \mathcal{L}^2(0,t,\mathcal{L}^2(0,1; \mathbb{R}^n))\times$ $\mathcal{L}^2(0,t,\mathcal{L}^2(0,1; \mathbb{R}^q))$ be a mild solution pair to (\ref{abstract}). Pick $\{d^k\}_{k\in \mathbb{N}}\subset \mathcal{C}^\infty_c(0,t, \mathcal{ L}^2(0,1;\mathbb{R}^q))$ such that, one has:
		{
		\begin{equation}\label{sequence1}
			d^k \xrightarrow[k\xrightarrow{}\infty]{\mathcal{L}^2(0,t;\mathcal{L}^2(0,1;\mathbb{R}^q)) } d
		\end{equation}}} 
	{
	Since $\mathcal{D(A)}$ is dense in $\mathcal{L}^2(0,1;\mathbb{R}^n)$, then there exists a sequence $\{X^k_0\}_{k\in \mathbb{N}}\subset \mathcal{D(A)}$ such that
		\begin{equation}\label{sequence2}
			X^k_0 \xrightarrow[k\xrightarrow{}\infty]{\mathcal{L}^2(0,1;\mathbb{R}^n)} X_0
		\end{equation}}
		\\
	 We know that a strong solution pair to (\ref{cls}) is also a mild solution pair. Moreover, $\mathcal{A}$ is ${\rho}$-non-accretive, $X_0 \in \overline{\mathcal{D(A)}}$, and since $(X,d), \{(X^k,d^k)\}_{k\in \mathbb{N}}$ are mild solutions to (\ref{cls}). Therefore, from \cite[Theorem 4.1, page 130]{barbu2010nonlinear}, it holds:
		\begin{equation*}
			\begin{aligned}
				\norm{X(t)-X^k(t)}_{{{\mathcal{L}^2_n}}}= e^{{\rho} t}\norm{{X_0-X^k_0}}_{{{\mathcal{L}^2_n}}} \\+ \int^t_0 e^{{\rho} (t-\tau)}[X(\tau)-X^k(\tau), d(\tau)-d^k(\tau)]_s d\tau 
			\end{aligned}
		\end{equation*}
		{where for functions $x, y$ in real Banach spaces, $[\cdot, \cdot]_s$ is the directional derivative of the function $x\xrightarrow{} \norm{x}$ in the direction $y$}, defined by
		$$[x, y]_s:= \lim_{\lambda \xrightarrow[]{} 0} \frac{\norm{x+\lambda y}-\norm{x}}{\lambda}$$
		Using  \cite[Proposition 3.7, (iv)]{barbu2010nonlinear}, one has 
		\begin{equation}\begin{aligned}
				-\norm{d(\tau)-d^k(\tau)}_{{{\mathcal{L}^2_q}}}\leq[X(\tau)-X^k(\tau), d(\tau)-d^k(\tau)]_s\\\leq \norm{d(\tau)-d^k(\tau)}_{{{\mathcal{L}^2_q}}}
			\end{aligned}
		\end{equation}
		Then using the previous statement, one has:
		\begin{equation}\label{key}
			\begin{split}
				\norm{X(t)-X^k(t)}_{{{\mathcal{L}^2_n}}}\leq& e^{{\rho} t}\norm{{X_0-X^k_0}}_{{{\mathcal{L}^2_n}}} \\&+ \int^t_0 e^{{\rho} (t-\tau)} \norm{d(\tau)-d^k(\tau) }_{{{\mathcal{L}^2_q}}}d\tau 
			\end{split}
		\end{equation}
		Since the term $\norm{d(\tau)-d^k(\tau)}_{{{\mathcal{L}^2_q}}}$ is convergent as we can see in (\ref{sequence1}-(\ref{sequence2}), we have:
		\begin{equation}\label{limitconvergence}
			\begin{aligned}
				\lim_{k\xrightarrow[]{}\infty}\int^t_0 e^{{\rho} (t-\tau)} \norm{d(\tau)-d^k(\tau)}_{{{\mathcal{L}^2_q}}} d\tau \\\leq \int^t_0 e^{{\rho} (t-\tau)}\lim_{k\xrightarrow[]{}\infty} \norm{d(\tau)-d^k(\tau)}_{{{{\mathcal{L}^2_q}}}} d\tau
			\end{aligned}
		\end{equation}
		So taking the limit as $k\xrightarrow[]{}\infty$ in (\ref{key}), one has:
		\begin{equation}\label{limitinequality}
			\begin{aligned}
				\lim_{k\xrightarrow[]{}\infty} \norm{X(t)-X^k(t)}_{{{\mathcal{L}^2_n}}}\leq  e^{{\rho} t}\lim_{k\xrightarrow[]{}\infty}\norm{{X_0-X^k_0}}_{{{\mathcal{L}^2_n}}} \\+ \int^t_0 e^{{\rho} (t-\tau)}\lim_{k\xrightarrow[]{}\infty} \norm{d(\tau)-d^k(\tau)}_{{{\mathcal{L}^2_q}}} d\tau 
			\end{aligned}
		\end{equation}
		
		Thus, from  (\ref{sequence1}), (\ref{sequence2}), (\ref{limitconvergence}) and (\ref{limitinequality}), we can {infer} that
		\begin{equation}\label{limitXk}
			\lim_{k\xrightarrow[]{}\infty} X^k(t){\overset{\mathcal{L}^2(0,1;\mathbb{R}^n)}{=}}X(t)
		\end{equation}
		and the proof is concluded.
	\end{proof}
	\section{Stability Analysis and Control Design}\label{stabilityanalysis}
	This section contains results on the $\mathcal{L}^2$-stability analysis to achieve closed-loop exponential stability. This is done first by proposing sufficient conditions and then constructing a Lyapunov functional to derive those sufficient conditions in the form of functional inequalities. 
	\subsection{Sufficient Conditions}
	The following section presents the sufficient conditions for the solution to Problem \ref{solutionproblem} using a dissipation inequality. This is done by proving the following proposition:
	\begin{proposition}
		Assume that there exists a Fréchet differentiable functional $V: \mathcal{L}^2(0,1; \mathbb{R}^n) \xrightarrow[]{} \mathbb{R}_{\geq 0}$ and $c_1, c_2, c_3, \chi \in \mathbb{R}_{>0}$ such that for each $d\in \mathcal{L}^2(0,1;\mathbb{R}^q)$ and $\zeta \in \mathcal{D(A)}$.
		\begin{equation}\label{SC1}
			c_1\norm{\zeta} ^ 2_{{{{{\mathcal{L}^2_n}}}}}\leq V(\zeta)\leq  c_2\norm{\zeta} ^ 2_{{\mathcal{L}^2_n}}
		\end{equation}
		\begin{equation}\label{SC2}
			DV(\zeta)(\mathcal{A}\zeta +\mathcal{N}d) \leq - c_3 V(\zeta) + \chi^2\norm{d}^2_{{{\mathcal{L}^2_q}}}
		\end{equation}
		Let $(X,d)$ be a mild solution pair to (\ref{abstract}). Then, for all $t\in\dom  X$, one has:
		\begin{equation}\label{SC3}
			\begin{split}
				\norm{X(t)}_{{\mathcal{L}^2_n}} \leq &  e^{-\frac{c_3}{2}t}{\left(\frac{c_2}{c_1}\right )}^\frac{1}{2} \norm{{X_0}}_{{\mathcal{L}^2_n}}+\frac{\chi}{\sqrt{c_1}} \sqrt{\int^t_0\norm{d(\theta)}^2_{{\mathcal{L}^2_q}} d\theta}
			\end{split}
		\end{equation}
	\end{proposition}
		\begin{proof}
			First we show that the above results hold for all strong solution pairs to (\ref{abstract}). More precisely, we consider solution pair $\dom X\ni t \mapsto (X(t),d(t))$ to (\ref{abstract}) and assume that ${X(0)} \in \mathcal{D(A)}, d\in \mathcal{L}^1(\dom d; \mathcal{L}^2(0,1; \mathbb{R}^q))$. Then, since, as shown in the proof of Proposition \ref{accretprop}, $\mathcal{A}$ is $\rho$-non-accretive, one has that $(X,d)$ is a {strong} solution pair (this is proved in \cite{barbu2010nonlinear}, Theorem 4.14).
			More precisely, one has that $X\in \mathcal{C}^1(\dom X, \mathcal{L}^2(0,1;\mathbb{R}^n))$, and for all $t\in \dom X$:
			\begin{equation}\label{abstract2}
				\dot{X}(t)= \mathcal{A} X(t) + \mathcal{N} d(t)
			\end{equation}
			where $X(t) \in \mathcal{D(A)}$.
			Now, consider the following function:
			\begin{equation}
				\begin{array}{l}
					\mathcal{W}: \dom X \xrightarrow[]{} \mathbb{R}\\
					\qquad t\mapsto (V \circ X)(t)
				\end{array}
			\end{equation}
			Then, since { $V: \mathcal{L}^2(0,1;\mathbb{R}^n) \xrightarrow[]{}\mathbb{R}_{\geq 0}$} is Fréchet differentiable everywhere and $X: \dom X\xrightarrow[]{} \mathcal{L}^2(0,1;\mathbb{R}^n)$ is differentiable almost everywhere, it follows that for almost all $t\in \dom{X}$:
			$$\dot{\mathcal{W}}(t)= DV(X)\dot {X}(t)$$
			which thanks to (\ref{abstract2}) yields for almost all $t\geq 0$
			$$ {\dot{\mathcal{W}}}(t)= DV(X)(\mathcal{A} X(t) + \mathcal{N} d(t))$$
			Thus, using (\ref{SC2}) one gets for almost all $t\in \dom X$
			$$\dot {\mathcal{W}}(t)\leq -c_3\mathcal{W}(t) + \chi^2\norm{d(t)}^2_{{{\mathcal{L}^2_q}}}$$
			Therefore, since $\mathcal{W}$ is continuous on $\dom X$, {from the comparison principle (\cite[Page 102]{8}) we have that: } we have:
			$$\mathcal{W}(t)\leq e^{-c_3t}\mathcal{W}(0)+ \chi^2\int^t_0 e^{-c_3(t-\theta)} \norm{{{d(\theta)}}}^2_{{{\mathcal{L}^2_q}}} d\theta$$
			At this stage, notice that for all $t \in \dom X$, one has:
			$$\int^t_0 e^{-c_3(t-\theta)} \norm{{d(\theta)}}^2_{{{\mathcal{L}^2_q}}} d\theta \leq \int^t_0 \norm{d(\theta)}^2_{{{\mathcal{L}^2_q}}} d\theta$$
			which allows one to conclude that for all $t \in \dom X$
			$$\mathcal{W}(t)\leq e^{-c_3t}\mathcal{W}(0)+ \chi^2\int^t_0  \norm{d(\theta)}^2_{{{\mathcal{L}^2_q}}} d\theta$$
			Finally by using (\ref{SC1}), it follows that for almost all $t\in \dom X$ 
			\begin{equation}\label{SC4}
				\norm{X(t)}_{{{\mathcal{L}^2_n}}} \leq  e^{-\frac{c_3}{2}t}\sqrt{\frac{c_2}{c_1}} \norm{{X_0}}_{{{\mathcal{L}^2_n}}}+\frac{\chi}{\sqrt{c_1}} \sqrt{\int^t_0\norm{d(\theta)}^2_{{{\mathcal{L}^2_q}}} d\theta}
			\end{equation}
			Now we conclude the proof by showing that the above bound holds also for mild solution pairs to (\ref{abstract}). Let $(X,d)$ be any solution pair. By applying Proposition 3, there exists a sequence of strong solution pairs $\{(X^k, d^k)\}_{k\in \mathbb{N}}$ such that (\ref{xseq}) and (\ref{dseq}) hold. Then,  for all $ k\in \mathbb{N}$, thanks to (\ref{SC4}), one has for all $t\in \dom X$
			$$\norm{X^k(t)}_{{{\mathcal{L}^2_n}}} \leq  e^{-\frac{c_3}{2}t}\sqrt{\frac{c_2}{c_1}} \norm{X^k_0}_{{{\mathcal{L}^2_n}}}+\frac{\chi}{\sqrt{c_1}} \sqrt{\int^t_0\norm{d^k(\theta)}^2_{{{\mathcal{L}^2_q}}}} d\theta$$
			Taking the limit for $k\xrightarrow[]{} \infty$, due to (\ref{xseq})-(\ref{dseq}), one has for all $t\in \dom X$
			$$\norm{X(t)}_{{{\mathcal{L}^2_n}}} \leq  e^{-\frac{c_3}{2}t}\sqrt{\frac{c_2}{c_1}} \norm{{X_0}}_{{{\mathcal{L}^2_n}}}+\frac{\chi}{\sqrt{c_1}} \sqrt{\int^t_0\norm{d(\theta)}^2_{{{\mathcal{L}^2_q}}} d\theta}$$
			This concludes the proof.
		\end{proof}
	\begin{remark}
		Proposition 4 provides sufficient conditions for input-to-state stability for the closed-loop system in the form of a functional inequality. This provides an elegant generalization to abstract dynamical systems of the well-known ISS dissipation inequality for finite-dimensional nonlinear systems; see, e.g \cite{sontag2008input}. It is interesting to observe that the gradient of $V$ is replaced in (\ref{SC2}) by the
		Fr\'echet derivative. 
	\end{remark}
	\subsection{Quadratic Conditions}
	Let us define the following global sector condition which will be useful in the upcoming Lyapunov analysis computations. 
	\begin{lemma} \cite[page 41]{2}
		For all $\nu \in \mathbb{R}^m$, the nonlinearity $\phi (\nu)$ satisfies the following inequality:
		\begin{equation}\label{gsc}
			\phi(\nu)^\top T(\phi(\nu) + \nu) \leq 0
		\end{equation} for any diagonal matrix T $\in \mathbb{D}^m_p$.
	\end{lemma}
	The following theorem provides sufficient conditions in the form of matrix inequalities under which Problem \ref{solutionproblem} admits a feasible solution. 
	\begin{theorem}\label{lyapth}
		If there exist $ P \in \mathbb{D}_p^n$, $T \in \mathbb{D}_p^m$, $\mu, \chi, {\alpha} \in \mathbb{R}_{>0}$, and $\Gamma\in \mathbb{S}^n_p$ such that the following hold:
		\begin{equation}\label{lmi1}
		\begin{pmatrix} H_{cl}^\top P\Lambda H_{cl}- e^{-\mu}P\Lambda  &  H_{cl}^\top P\Lambda B-K^\top T\\
				* & B^\top P\Lambda B-2T \end{pmatrix}\leq 0
		\end{equation}
		\begin{equation}\label{lmi2}
			\begin{pmatrix} \Gamma & PN \\ * &\chi^2 I \end{pmatrix}\geq 0
		\end{equation}
		\begin{equation}\label{lmi3}
			P(\alpha I-\mu \Lambda)+\Gamma \leq 0
		\end{equation}
		Then, $K$ solves Problem \ref{solutionproblem} and in particular (\ref{lmiprob}) holds with
		\begin{equation}\label{lyapcoefficients}
			\begin{array}{lc}
				\omega = \frac{\alpha}{2}, & \kappa = \sqrt{\frac{\lambda_{\max}(P)}{\lambda_{\min}(P)}} e^{\frac{\mu}{2}} \\ \gamma= \frac{\chi}{\lambda_{\min}(P)} e^{\frac{\mu}{2}}
				& 
			\end{array}
		\end{equation}
	\end{theorem}
	\begin{proof}
		Similarly as in \cite{1}, consider the following Lyapunov functional
		\begin{equation}
			\begin{split}
				V:& \mathcal{L}^2(0,1; \mathbb{R}^n) \xrightarrow[]{} \mathbb{R} \\&
				X\mapsto \int^1_0 e^{-\mu z}\langle X(z), PX(z)\rangle_{\mathbb{R}^n} dz
			\end{split}
		\end{equation}
		{ with the same $\mu$ defined in (\ref{Lmuspace})} and observe that for each $X\in \mathcal{L}^2(0,1;\mathbb{R}^n)$, one has 
		\begin{equation}\label{Vineq}
			c_1\norm{X}^2_{{{\mathcal{L}^2_n}}}\leq V(X)\leq c_2 \norm{X}^2_{{{\mathcal{L}^2_n}}}    
		\end{equation}
		where $c_1:= e^{-\mu}\lambda_{\min}(P)$ and $c_2:= \lambda_{\max}(P)$ are strictly positive.
		 As done in \cite{ferrante2019boundary},  for each $X\in \mathcal{D(A)}$, $d\in \mathcal{L}^2(0,1;\mathbb{R}^q)$ one has 
	\begin{equation*}
		\begin{split}
			DV(X)(\mathcal{A}X+\mathcal{N}d)= \int^1_0 e^{-\mu z}  (& -2X_z(z)^\top \Lambda P X(z)\\&+ 2d(z)^\top N^\top P X(z) ) dz
		\end{split}
	\end{equation*}
		Since $P, \Lambda \in \mathbb{D}^n_p$, one has that 
		\begin{equation*}\begin{split}
			&\int^1_0 -2e^{-\mu z} X_z(z)^\top \Lambda P X(z)dz\\&= -\int^1_0 e^{-\mu z}\frac{d}{dz}\left(X(z)^\top P \Lambda X(z)\right) dz
		\end{split}\end{equation*}
		Using integration by parts, the following holds
	{\begin{equation*}\begin{split}
			&DV(X)(\mathcal{A}X+\mathcal{N}d)= -e ^{-\mu z}X(z)^\top P \Lambda X(z)\Bigr|^1_0 \\&-\mu \int^1_0 e ^{-\mu z}X(z)^\top P \Lambda X(z) dz +\int^1_0 2 e ^{-\mu z} d(z)^\top N^\top P X(z) dz
		\end{split}\end{equation*}}
	Since $X\in \mathcal{D(A)}$, one gets 
		\begin{equation*}
		\begin{split}
		&	DV(X)(\mathcal{A}X+\mathcal{N}d)= \\
			&\mathcal{X}^\top
			\begin{pmatrix}
				H_{cl}^\top P\Lambda H_{cl}- e^{-\mu}P\Lambda  &  H_{cl}^\top P\Lambda B\\
				* & B^\top P\Lambda B 
			\end{pmatrix}
			\mathcal{X}\\&
			+\int^1_0 e^{-\mu z}\ \begin{pmatrix}X(z)\\d(z)
			\end{pmatrix}^\top \begin{pmatrix}
				-\mu P \Lambda  & PN \\
				* & 0 
			\end{pmatrix}\begin{pmatrix}X(z)\\d(z)
			\end{pmatrix} dz
		\end{split}\end{equation*}
			where $\mathcal{X}\coloneqq \begin{pmatrix}
			X(1) \\
			\phi(KX(1)) 
		\end{pmatrix}$. Similarly as in \cite{shreim2020design}, after we introduce the global sector condition found in (\ref{gsc}) one has 
		\begin{equation*}
		\begin{split}
			&DV(X)(\mathcal{A}X+\mathcal{N}d)\leq\\&
		\mathcal{X}^\top 
			\begin{pmatrix}
				H_{cl}^\top P\Lambda	H_{cl}- e^{-\mu}P\Lambda  &  	H_{cl}^\top P\Lambda B-K^\top T\\
				* & B^\top P\Lambda B-2T 
			\end{pmatrix}
			\mathcal{X}\\&
			+\int^1_0 e^{-\mu z}\begin{pmatrix}X(z)\\d(z)
			\end{pmatrix}^\top \begin{pmatrix}
				-\mu P \Lambda  & PN \\
				* & 0 
			\end{pmatrix}\begin{pmatrix}X(z)\\d(z)
			\end{pmatrix} dz
		\end{split}	
	\end{equation*}
		where $T\in \mathbb{D}^m_p$. From (\ref{lmi2}) one has 
		\begin{equation*}\begin{split}
			&DV(X)(\mathcal{A}X+\mathcal{N}d)\leq \\
			&\mathcal{X}^\top 
			\begin{pmatrix}        	H_{cl}^\top P\Lambda 	H_{cl}- e^{-\mu}P\Lambda  &  	H_{cl}^\top P\Lambda B-K^\top T\\
				* & B^\top P\Lambda B-2T 
			\end{pmatrix}
		\mathcal{X}\\&
			+\int^1_0 e^{-\mu z}\begin{pmatrix}X(z)\\d(z)
			\end{pmatrix}^\top \begin{pmatrix}
				-\mu P \Lambda + \Gamma & 0 \\
				* & \chi^2 I 
			\end{pmatrix}\begin{pmatrix}X(z)\\d(z)
			\end{pmatrix}dz
		\end{split}\end{equation*}
		Finally, using (\ref{lmi1}) and (\ref{lmi3}) we have
		\begin{equation}
			DV(X)(\mathcal{A}X+\mathcal{N}d)\leq -\alpha V(X) + \chi^2\norm{d}^2_{{{\mathcal{L}^2_q}}}
		\end{equation}
		which reads as (\ref{SC2}). Hence, the proof is concluded.\end{proof}
		{\begin{remark}
		    We emphasize that the result presented in Theorem \ref{lyapth} is global. However, the feasibility of (\ref{SC1}) is not always guaranteed. This is commonly seen in the literature of saturated control (see \cite[Chapter3]{2}).
		\end{remark}}
	\subsection{Control Design}\label{Controldesign}
	
	Theorem \ref{lyapth} enables to recast the solution to Problem \ref{solutionproblem} as the feasibility problem of some matrix inequalities, i.e. (\ref{lmi1})-(\ref{lmi2})-(\ref{lmi3}). However, those conditions are nonlinear in the variables $P, K, \chi, \mu$ and $\alpha$. As such, Theorem \ref{lyapth} cannot be used directly to get a numerically tractable solution to Problem \ref{solutionproblem}. The result given next, provides sufficient conditions in a form that is more advantageous from a numerical standpoint. 
	\begin{corollary}\label{Schur}
		If there exist $Q\in \mathbb{D}^{n}_p, S\in \mathbb{D}^m_p, \widehat{\Gamma} \in \mathbb{S}^{n}_p$, $\mu, \alpha \in \mathbb{R}_{>0}$ and $W\in \mathbb{R}^{m \times n}$ such that 
		\begin{equation}\label{lmi4}
			\begin{pmatrix} -Q\Lambda^{-1} & HQ + BW & BS\\ * & -e^{-\mu} \Lambda Q & -W^\top \\ * & * & -2S\end{pmatrix} \leq 0
		\end{equation}
		\begin{equation}\label{lmi5}
			\begin{pmatrix} \widehat{\Gamma} & N \\
				* & I \end{pmatrix} \geq 0
		\end{equation}
		\begin{equation}\label{lmi6}
			Q(\alpha I - \mu\Lambda) + \widehat{\Gamma}\leq 0
		\end{equation}
		Then, $K=WQ^{-1}$ solves Problem \ref{solutionproblem}. In particular, (\ref{lmiprob}) holds with $\omega$ and $\kappa$ defined as in (\ref{lyapcoefficients}) and 
		\begin{equation}\label{gammacond}
			\gamma = \sqrt{\lambda_{\max}(Q)} e^{\frac{\mu}{2}}
		\end{equation}
	\end{corollary}
	\begin{proof}
		Applying the Schur complement lemma to (\ref{lmi1}) gives
		\begin{equation*}
			\begin{pmatrix}
				-\Lambda^{-1} P & PH +PBK & PB\\
				* & -e^{-\mu} P\Lambda & -K^\top T\\
				* & * & -2T
			\end{pmatrix}\leq 0
		\end{equation*}
	which is equivalent to
	\begin{equation*}
		C^\top \begin{pmatrix}
				-\Lambda^{-1} P & PH +PBK & PB\\
			* & -e^{-\mu} P\Lambda & -K^\top T\\
			* & * & -2T
		\end{pmatrix}C
	\end{equation*}
where $C= \begin{pmatrix}P^{-1} & 0 & 0\\
*& P^{-1}& 0\\
* & * & T^{-1}\end{pmatrix}$
which gives
\begin{equation*}
	\begin{pmatrix}
		-P^{-1}\Lambda^{-1} & HP^{-1}+ BKP^{-1} & BT^{-1}\\ * & -e^{-\mu}\Lambda P^{-1} & -P^{-1} K^\top\\ * & * & -2T^{-1}
	\end{pmatrix} \leq 0
\end{equation*}
Then, by setting $P^{-1}= Q$, $T^{-1}=S$ and $W= KP^{-1}$, we have that the previous inequality is equivalent to (\ref{lmi4}).
		In \cite[Corollary 1]{ferrante2019boundary}, it is shown that (\ref{lmi2}) and (\ref{lmi3}) are respectively equivalent to the linear inequalities (\ref{lmi5}) and (\ref{lmi6}) with $P^{-1}= Q$, $\widehat{\Gamma}= Q\Gamma Q$, and $\chi=1$.
	\end{proof}
	\begin{remark}
		It can be easily shown that the conditions in Corollary \ref{Schur} are actually equivalent to those in Theorem \ref{lyapth} in terms of feasibility. As such, Corollary \ref{Schur} does not introduce any additional conservatism. 
	\end{remark}
	In the formulation of Problem \ref{solutionproblem}, no specific requirements on the scalar $\gamma$ are considered. On the other hand, it is obvious that to minimize the effect of the exogenous input $d$ on the closed-loop system, the controller gain $K$ should be designed so that (\ref{lmiprob}) holds with a minimal $\gamma$. This goal can be achieved by considering the following optimization problem
	\begin{equation}\label{optimprob}
		\begin{array}{c}
			\underset{Q,W,\mu, \alpha, c}{\inf} c\\
			\text{s.t: } (\ref{lmi4}), (\ref{lmi5}), (\ref{lmi6}), Q\in \mathbb{D}^{n_p}_p, \mu >0, \alpha>0, Q-cI\leq 0. 
		\end{array}
	\end{equation}
	{It can be seen through equation (\ref{gammacond}) that the variable $\gamma$ is directly proportional to the square root of the maximum eigen value of matrix $Q$. Therefore, minimizing $c$ is equivalent to minimizing $\lambda_{\max (Q)}$.} One can note that (\ref{lmi4}), (\ref{lmi5}) and (\ref{lmi6}) are nonlinear in the decision variables $\mu$ and $\alpha$. In fact, we select the scalars $\mu$ and $\alpha$ via a grid search. {In other words,we set $\mu$ and $\alpha$ as arrays of appropriate values and resolution, and then running the code to generate a 3-dimensional diagram representing the feasible regions of the solution with respect to the pair $(\mu, \alpha)$.}

	\section{Numerical Example}\label{numerical}
	To solve initial-boundary value problem for (\ref{cls}), numerical integration of hyperbolic PDEs is performed via the use of the Lax-Friedrichs (Shampine’s two-step
variant) scheme implemented in \textsc{Matlab}
 by Shampine \cite{12}. YALMIP package in \textsc{Matlab} is used to solve the LMIs \cite{lofberg2004yalmip}.  Consider the example presented in \cite{shreim2020design} modified to account for the presence of in-domain disturbances. Specifically, we consider the following system for all $(t,z)\in \mathbb{R}_{\geq 0}\times (0,1)$:
	\begin{equation*}\begin{split}
		&X_t(t,z) + \begin{pmatrix}1 & 0\\0 & \sqrt{2}\end{pmatrix} X_z(t,z) = \begin{pmatrix} 1\\1\end{pmatrix} d(t,z)\\
		&X(t,0)= \begin{pmatrix}0.25 & 0\\-1 & 0.25 \end{pmatrix}X(t,1) + \begin{pmatrix}1 & 0\\0 & 1\end{pmatrix} u(t) \qquad \forall t \in \mathbb{R}_{\geq 0}
	\end{split}
	\end{equation*}
	We consider the solution to Problem \ref{solutionproblem} obtained by solving (\ref{optimprob}), via a line search on the scalars $\alpha$ and $\mu$. Fig \ref{bidimension} represents the set of feasible values of (\ref{optimprob}) of the pair $(\mu,\alpha)$. As in \cite{ferrante2019boundary}, we have that the feasible values of $\mu$ decreases as $\alpha$ increases. {Then, as seen in the figure, we choose $\mu=1$, $\alpha= 0.5$ in order to guarantee a feasible solution to our problem.}
	\begin{figure}
		\centering
		\includegraphics[scale=0.4]{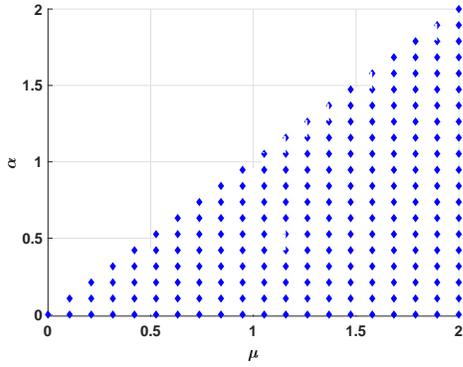}
		\caption{Feasible (diamond) pairs $(\mu,\alpha)$}
		\label{bidimension}
	\end{figure}
	For this example, the solution to (\ref{optimprob}) yields:
\begin{equation*}\begin{split}
		&Q= \begin{pmatrix} 12.5 & 0\\0 & 82 \end{pmatrix}, \Gamma=  \begin{pmatrix} 4.07 & 0.2\\0.19 & 36.3 \end{pmatrix}, K=\begin{pmatrix} -0.24 & 0\\0.33 & -0.08\end{pmatrix}\end{split}\end{equation*}
	Consider the following disturbance {defined over $t\in [0,25]$}:
	\begin{equation}\label{dist}
		(t,z)\mapsto d(t,z):= 5\begin{pmatrix} \sin(zt)\\ \cos(zt) \end{pmatrix}
	\end{equation}
	the initial condition: 
	\begin{equation*}
		(0,1) \ni z \mapsto X_0(z) = 10 \begin{pmatrix} \cos(4\pi z)-1\\\cos(2\pi z)-1\end{pmatrix}  
	\end{equation*} 
	and the saturation level $u_{\max}=\begin{pmatrix}0.3\\0.3\end{pmatrix}$.
	In Fig \ref{closedvsopen}, we report the evolution of the $\mathcal{L}^2$-norm of the closed-loop system state compared to that of the open-loop ($K=0$) in response to the disturbance (\ref{dist}). As expected, the gain $K$ designed via the proposed sufficient conditions provides better disturbance reduction and convergence rate than that in the open-loop case. 
	\begin{figure}
		\centering
		\includegraphics[scale=0.4]{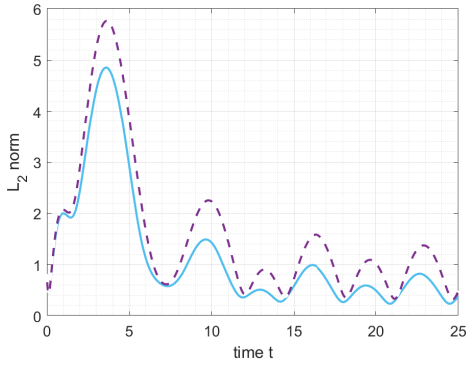}
		\caption{Time-evolution of the spatial norm $\mathcal{L}^2$ of $X(t,\cdot)$ in closed-loop (solid-line) and open-loop (dashed-line)}
		\label{closedvsopen}
	\end{figure}
	\begin{remark}
		In Fig \ref{saturation}, one can see the saturation levels of under which our controller perform. Thus, those results are in fact reflecting the behavior of this stabilizing design for the controller acting on the hyperbolic system (\ref{systemmm}).
	\end{remark}
	\begin{figure}
		\centering
		\includegraphics[scale=0.4]{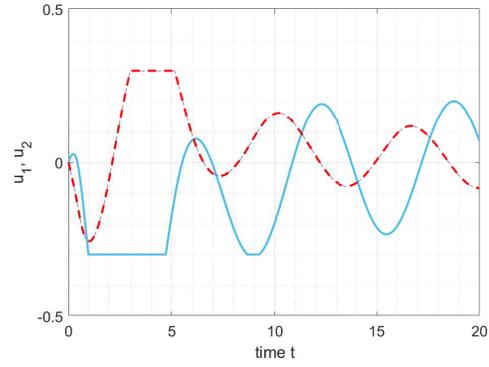}
		\caption{Time-evolution of $\sigma(K_1X_1(t,1))$ (solid-line) and $\sigma(K_2X_2(t,1))$ (dashed-line) with respect to time}
		\label{saturation}
	\end{figure}
	{Finally, Figure \ref{rightvsleft} shows the same time evolution of the $\mathcal{L}^2$- spatial norm of the closed-loop state $X(t)$ in comparison with that of the right-hand side of the dissipation inequality (\ref{SC3}) $$e^{-\frac{c_3}{2}t}{\left(\frac{c_2}{c_1}\right )}^\frac{1}{2} \norm{{X_0}}_{{\mathcal{L}^2_n}}+\frac{\chi}{\sqrt{c_1}} \sqrt{\int^t_0\norm{d(\theta)}^2_{{\mathcal{L}^2_q}} d\theta}$$ Looking at the plot, the formal computation is consistent with the stability result.}
	
	\begin{figure}
	    \centering
	    \includegraphics[scale=0.4]{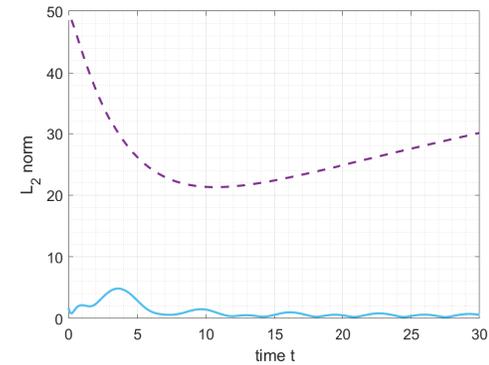}
	    \caption{Time-evolution of the left (solid-line) vs. right(dashed-line)-hand side of dissipation inequality (35)}
	    \label{rightvsleft}
	\end{figure}
	\section{Conclusion}\label{conclusion}
	Well-posedness and the global exponential stability of a class of 1D hyperbolic equations have been studied. The PDE under consideration was the result of a perturbed hyperbolic system in the presence of an in-domain exogenous disturbance connected in a feedback loop with a saturated nonlinear control law. The controller acted on the boundary condition. The well-posedness was investigated under the techniques of nonlinear semi-group theory proving the existence and uniqueness of a mild solution pair. Then, the approximation of this solution to a strong solution pair was established. Furthermore, sufficient conditions for the exponential stability  have been derived in the form of linear matrix inequalities using Lyapunov theory for infinite dimensional systems. Semi-definite programming tools
	were used to design the controller to minimize the effect of the disturbance and boundary nonlinearity on the $\mathcal{L}^2$ norm of the state. Numerical simulations were used to illustrate the effectiveness of the
	proposed control design strategy in an example.
	
	This article opens the horizon for some interesting questions. In particular, one can extend the research towards other classes of Lyapunov functions, as those considered in \cite{ahmadi2016dissipation} and to compare the consequent constraints with those present in Theorem \ref{lyapth}. Considering the extension of this application to the design of an observer is also a possible research track. 
\appendix
	\section{Auxiliary Results}	
	{\begin{definition}
		Let {$Y_1$ and $Y_2$} be linear normed spaces, $U$ be an open subset of {$Y_1$, $f: U \xrightarrow[]{} Y_2$}, and $x\in U$. We say that $f$ is Fréchet differentiable at $x$ if there exists $L\in \mathcal{L}(Y_1,Y_2)$ such that
		\begin{equation}
			\lim_{h\xrightarrow[]{} 0} \frac{\norm{f(x+h)-f(x) -Lh}_{Y_2}}{\norm{h}_{Y_1}}= 0
		\end{equation}
		In particular, $L$ is the Fréchet derivative of $f$ at $x$ and is denoted by $Df(x)$. When $Y_1=\mathbb{R}$, we denote $\dot{f}(x)= \lim_{h\xrightarrow[]{}0}\frac{f(x+h)- f(x)}{h}$
	\end{definition}
	\begin{lemma}
		Let $\Phi \in \mathcal{C}^0(0,1;\mathbb{R})$, $P \in \mathbb{D}^n_p$ and $\mathcal{L}^2(0,1;\mathbb{R}^n)$ be endowed with its standard inner product. Consider the following functional
		\begin{equation*}
			\begin{split}
				V:& \mathcal{L}^2(0,1;\mathbb{R}^n) \xrightarrow{} \mathbb{R}\\
				&X\mapsto V(X)\coloneqq \int^1_0 \Phi(z)\langle P X(z),X(z) \rangle dz
			\end{split}
		\end{equation*}
		Then, $V$ is Fr\'echet differentiable on $\mathcal{L}^2(0,1;\mathbb{R}^n)$ and in particular, for each $X,h\in \mathcal{L}^2(0,1;\mathbb{R}^n)$
		\begin{equation*}
			DV(X) h = 2 \langle \Phi P X, h\rangle_{{{\mathcal{L}^2_n}}}
		\end{equation*}
	\end{lemma}
	\begin{proof}
		For any $X,h \in \mathcal{L}^2(0,1;\mathbb{R}^n)$, one has
		\begin{equation*}\begin{split}
				V(X+h)- V(X)= &\int^1_0 \Phi(z)(\langle h (z), Ph(z)\rangle_{\mathbb{R}^n}\\
				&+ 2\langle X(z), Ph(z)\rangle _{\mathbb{R}^n})dz\\
				&\leq \lambda_{\max} (P) \norm{\Phi}_{\infty} \norm{h}^2 \\&+ 2\langle \Phi P X, h\rangle_{{{\mathcal{L}^2_n}}}
			\end{split}
		\end{equation*}
		Thus, it follows that
		$$\lim_{\norm{h}_{{{\mathcal{L}^2_n}}}\xrightarrow{} 0} \frac{|V(X+h)-V(X)- 2\langle X, \Phi Ph\rangle_{{{\mathcal{L}^2_n}}}|}{\norm{h}_{{{\mathcal{L}^2_n}}}}= 0$$
		This concludes the proof.
\end{proof}}

	\bibstyle{plain}        
	\bibliography{autosam}           

\end{document}